\documentclass[11pt,oneside,leqno]{amsart}

\usepackage{amsmath,amsfonts,amssymb}
\usepackage{amsxtra}
\usepackage{amsopn}
\usepackage{amsmath,amsthm,amssymb}
\usepackage{amscd}
\usepackage{pifont}
\usepackage{amsfonts}
\usepackage{latexsym}
\usepackage{verbatim}
\usepackage{pb-diagram}
\usepackage{graphicx}

\theoremstyle{plain}

\newtheorem{rem}{Remark}

\newtheorem*{mt*}{Main Theorem}

\theoremstyle{plain}
\newtheorem{theorem}{Theorem}[section]

\newcommand\e{{\varepsilon}}

\newcommand\pa[1]{\partial_{#1}}

\title{Ricci solitons of special Lorentzian Lie groups with a four-dimensional isometry group}

 \author{Giovanni Calvaruso}
\address{Giovanni Calvaruso: Dipartimento di Matematica e Fisica \lq\lq E. De Giorgi\rq\rq \\
Universit\`a del Salento\\
Prov. Lecce-Arnesano \\
73100 Lecce\\ Italy.}
\email{giovanni.calvaruso@unisalento.it}

\author{Lorenzo Pellegrino}
\address{Lorenzo Pellegrino: Dipartimento di Matematica e Fisica \lq\lq E. De Giorgi\rq\rq \\
Universit\`a del Salento\\
Prov. Lecce-Arnesano \\
73100 Lecce\\ Italy.}
\email{lorenzo.pellegrino@unisalento.it}

\author{Amirhesam Zaeim}
\address{Amirhesam Zaeim: Department of Mathematics\\
 Payame Noor University (PNU)\\
 P.O. Box 19395-4697, Tehran\\
  Iran}
  \email{zaeim@pnu.ac.ir}



\begin{document}


\maketitle

\begin{abstract}
In the framework of the study of homogeneous Lorentzian three-manifolds, we consider here the only class of examples which admit a four-dimensional group of isometries but are neither Lorentzian Bianchi-Cartan-Vranceanu spaces nor plane waves. We  obtain an explicit description in global coordinates of these special homogeneous Lorentzian manifolds. We then prove that all such examples are non-gradient expanding Ricci solitons.
\end{abstract}

\section{Introduction}
\setcounter{equation}{0}

The study of homogeneous  three-manifolds is a relevant topic, which attracted the interest of several researchers with regard to many different aspects of their geometric properties. A  three-dimensional  nondegenerate metric is (up to reversing it \cite{O'N}) either Riemannian or Lorentzian. Usually, the Lorentzian settings allow to a wider spectrum of properties and behaviours, that have not  any Riemannian counterpart. A non-exhaustive list of properties, with respect to which homogeneous Lorentzian three-manifolds show a geometry richer than the Riemannian ones, is given by:  Einstein-like metrics \cite{AGV,Cgd, Ciran}, homogeneous Ricci solitons \cite{Cerbo,isr, Corr}, homogeneous pseudo-Riemannian structures  \cite{CFGV,CZ}.

It is well known that the isometry group $\text{Iso}(M,g)$  of  a three-dimensional pseudo-Riemannian manifold  $(M,g)$ is at most six-dimensional. The case where
$\mathrm{dim}(\text{Iso}(M,g))= 6$ characterizes manifolds of constant sectional  curvature. Moreover, there exist no  three-dim\-ensional pseudo-Riemannian manifolds with a five-dimensional isometry group. It is then natural to consider three-dimensional pseudo-Riemannian manifolds $(M,g)$ admitting a four-dimensional isometry group. Observe that in these cases, $\text{Iso}(M,g)$ acts transitively and so, $(M,g)$ is a homogeneous pseudo-Riemannian manifold.

As it often occurs, the Lorentzian case is richer and subtler with respect to its Riemannian analogue.
In fact, {\em Bianchi-Cartan-Vranceanu spaces}  (BCV spaces, for short) \cite{B1,B2,C,V} exhaust the description of Riemannian three-manifolds with a four-dimensional isometry group, but there are different classes of homogeneous Lorentzian three-manifolds with such a property, as it can be seen by comparing the classification of (non-isometric) homogeneous structures in the Riemannian \cite{CFGV} and the Lorentzian case \cite{CZ}.  More precisely, following \cite{CZ}, three-dimensional Lorentzian Lie groups admitting a four-dimensional isometry group, fall between one of the following not overlapping classes, corresponding to different possibilities for the self-adjoint structure operator $L$ of their Lie algebra:
\begin{itemize}{}
\item[(I)] {\em Lorentzian Bianchi-Cartan-Vranceanu spaces}, for which $L$ is diagonal. They have a structure similar to Riemannian BCV (although the Lorentzian class contains more cases than its Riemannian analogue). Lorentzian BCV spaces possess several remarkable geometric properties. In particular, they are naturally reductive spaces and are defined by submersions over a pseudo-Riemannian surface of constant curvature \cite{CP1}.

\item[(II)] an  {\em exceptional example},  for which the minimal polynomial of $L$ has a non-vanishing  double root.  At the Lie algebra level, this Lorentzian Lie group
 is described by
\begin{equation}\label{lie}
[u_1,u_2]=\mu u_3, \quad [u_2,u_3]=\mu u_2,\quad [u_3,u_1]=\mu u_1 +\varepsilon u_2, \qquad \mu \neq 0, \;  \varepsilon=\pm 1,
\end{equation}
where $\{u_1,u_2,u_3\}$ is a pseudo-orthonormal basis with
\begin{equation}\label{glie}
\langle u_1, u_2 \rangle = \langle u_3, u_3 \rangle =1.
\end{equation}
Equations \eqref{lie} and \eqref{glie} describe a one-parameter family of left-invariant Lorentzian metrics on $G=\widetilde{SL}(2,\mathbb R)$. These examples are still naturally reductive  \cite{CZ} but have not a structure similar to Lorentzian BCV spaces. Being naturally reductive, they are g.o. spaces (all their maximal geodesics are orbits of a one-parameter group of isometries, see for example \cite{CC}) and so, they are geodesically complete \cite{CWZ}.

\item[(iii)]  some {\em homogeneous plane waves},  for which the minimal polynomial of $L$ has $0$ as a triple root.   Homogeneous plane waves are a well known topic, whose study goes back to the work \cite{BL} in the framework of Theoretical Physics. These examples are not naturally reductive.
\end{itemize}

By definition,  a pseudo-Riemannian manifold $(M,g)$, together with a smooth vector field $X$, is a \emph{Ricci soliton} if
\begin{equation}\label{solit}
\mathcal{L}_Xg+\varrho=\lambda g,
\end{equation}
where $\mathcal{L}_X$ denotes the Lie derivative in the direction of $X$, $\varrho$ the Ricci tensor and $\lambda$ is a real number. A Ricci soliton is either \emph{shrinking}, \emph{steady} or \emph{expanding}, depending on whether $\lambda>0$, $\lambda=0$ or $\lambda<0$, respectively.

A Ricci soliton is said to be {\em gradient} if $X$ is the gradient, with respect to $g$, of some smooth function $f$ (called {\em potential}). In such a case, equation \eqref{solit} becomes
$$
\textit{Hess} (f)  + \varrho  =  \lambda g,
$$
where $\textit{Hess}$ denotes the Hessian. 

The Ricci soliton equation~\eqref{solit} is a special case of Einstein field equations. Moreover, Ricci solitons are the self-similar solutions of the {\em Ricci flow}.  They were first introduced in Riemannian settings  \cite{Ham}  and then extended to the pseudo-Riemannian case (in particular, to the Lorentzian one). We may refer to  the survey  \cite{Cao} for the Riemannian case, while some examples of  the study of  Ricci solitons in pseudo-Riemannian settings may be found in  \cite{AAC},\cite{BBGG},\cite{isr},\cite{Corr},\cite{Cosc}-\cite{CR},\cite{CZ1},\cite{Cas},\cite{Kh},\cite{PT} and references therein.

With respect to local coordinates, equation~\eqref{solit} yields an overdetermined system of second-order PDE, whose integration, when possible, provides a complete classification of solutions to the Ricci soliton equation  for the given class of metrics. 

In the case of a left-invariant Ricci soliton on a pseudo-Riemannian Lie group, the Ricci soliton equation \eqref{solit} translates into a system of algebraic equations. However, this appproach does not provide all solutions to the Ricci soliton equation. Left-invariant Ricci solitons for three-dimensional Lorentzian Lie groups were investigated in \cite{isr}. It may be observed that the curvature of a three-dimensional pseudo-Riemannian manifold is completely determined by its Ricci tensor, which plays a fundamental role in the Ricci soliton equation~\eqref{solit}.

We aim to obtain a complete classification of Ricci solitons for three-dimensional Lorentzian Lie groups, starting with the ones admitting a four-dimensional isometry group. Ricci solitons for three-dimensional plane waves (and, more in general, pp-waves) were investigated in \cite{BBGG}.  The contribution of the actual paper to this program is to solve completely the Ricci soliton equation for the exceptional examples described above by equations \eqref{lie} and \eqref{glie}.

 In order to do so, we shall first give an explicit description in global coordinates of the left-invariant Lorentzian metrics  determined by equations~\eqref{lie} and~\eqref{glie}. We then completely solve the Ricci soliton equation for such metrics. The main results are resumed in the following.

\begin{theorem}\label{main}
For any choice of the real parameter $\mu \neq 0$, the unimodular Lie group $G=\widetilde{SL}(2,\mathbb R)$, equipped with the left-invariant Lorentzian  metric corresponding to \eqref{glie},  is isometric to $\mathbb R^3$  equipped with the metric
\begin{equation}\label{g}
g=\frac{\varepsilon}{\mu} (e^{-2\mu x_3}-1) dx_1^2+2 dx_1 dx_2+2\mu x_2 dx_1dx_3+dx_3^2, \qquad \e =\pm 1.
\end{equation}
Every metric \eqref{g} is an expanding Ricci soliton,  satisfying the Ricci soliton equation~\eqref{solit} with $\lambda=-\frac 12 \mu ^2$ and $X=X^i \partial_i$, where
\begin{equation}\label{Xig1}
\left\{
\begin{array}{l}
X_1=  \frac{1}{\mu}   \left(k_1 e^{\sqrt{-\varepsilon \mu} x_1}+k_2 e^{-\sqrt{-\varepsilon \mu} x_1}+k_3 \right) , \\[7pt]
X_2=  - \frac{1}{\mu^2}   \left( k_1(\mu\sqrt{-\varepsilon\mu} x_2 -\varepsilon) e^{\sqrt{-\varepsilon \mu} x_1} -k_2(\mu\sqrt{-\varepsilon\mu} x_2 +\varepsilon) e^{-\sqrt{-\varepsilon \mu} x_1} \right) \\[3pt]
\qquad \quad+k_4 e^{-\mu x_3}, \\[7pt]
X_3=  \frac{\sqrt{-\varepsilon \mu} }{\mu^2}   \left(k_1 e^{\sqrt{-\varepsilon \mu} x_1}-k_2 e^{-\sqrt{-\varepsilon \mu} x_1} \right) -\frac\mu2
\end{array}
\right.
\end{equation}
if $\varepsilon\mu <0$, and
\begin{equation}\label{Xig2}
\left\{
\begin{array}{l}
X_1=  \frac{1}{\mu}   \left(k_1 \cos (\sqrt{\varepsilon \mu} x_1) +k_2 \sin (\sqrt{\varepsilon \mu} x_1) +k_3 \right) , \\[7pt]
X_2=  \frac{1}{\mu^2}   \left( k_1(\mu\sqrt{\varepsilon\mu} x_2  \sin (\sqrt{\varepsilon \mu} x_1)+\varepsilon \cos (\sqrt{\varepsilon \mu} x_1))  \right. \\[3pt]
\qquad \;\; \quad \quad \left. -k_2(\mu\sqrt{\varepsilon\mu} x_2  \cos (\sqrt{\varepsilon \mu} x_1)-\varepsilon \sin (\sqrt{\varepsilon \mu} x_1)) \right) +k_4 e^{-\mu x_3}, \\[7pt]
X_3=\frac{\sqrt{\varepsilon \mu} }{\mu^2} \left(-k_1 \sin (\sqrt{\varepsilon \mu} x_1) +k_2 \cos (\sqrt{\varepsilon \mu} x_1) \right) -\frac\mu2
\end{array}
\right.
\end{equation}
if $\varepsilon\mu >0$, for some real constants $k_1,k_2,k_3,k_4$. These Ricci solitons are not gradient.
\end{theorem}

The paper is organized in the following way. In Section~2 we shall prove the first part of Theorem~\ref{main}, providing the description in global coordinates for the exceptional example in the above classification of three-dimensional Lorentzian Lie groups with a four-dimensional isometry group and the needed information concernig its Levi-Civita connection and curvature. In Section~3 we shall then complete the proof of  Theorem~\ref{main},  solving the Ricci soliton equation~\eqref{solit} for all such metrics and proving that they are non-gradient Ricci solitons.

\section{Explicit coordinates and geometry of the exceptional example}
\setcounter{equation}{0}

The explicit form \eqref{g} for the metric tensors described by \eqref{lie} and \eqref{glie}, together with the needed information concerning the Levi-Civita connection and curvature of this family of Lorentzian metrics, are given in the following result, which shall also allows further investigations about the gometry of the exceptional example.

\begin{theorem}\label{geomg}
For any choice of real parameter $\mu \neq 0$, the Lie group $G=\widetilde{SL}(2,\mathbb R)$, equipped with the left-invariant Lorentzian metric corresponding to
 \eqref{lie} and \eqref{glie},  is isometric to $\mathbb R^3$ equipped with the metric \eqref{g}. Moreover,

\medskip
a) The levi-Civita connection of the metric $g$ described in \eqref{g} is completely determined by by the following possibly non-vanishing terms:
\begin{equation}\label{nablagcoord}
\begin{array}{l}
\nabla_{\pa1}\pa1=-\varepsilon \mu x_2  e^{-2\mu x_3}\pa2+ \varepsilon  e^{-2\mu x_3}\pa3,
\\[7pt]
\nabla_{\pa1}\pa2=\nabla_{\pa2}\pa1=-\frac 12  \mu ^2 x_2 \pa2+ \frac 12  \mu \pa3,
\\[7pt]
\nabla_{\pa1}\pa3=\nabla_{\pa3}\pa1=-\frac 12  \mu  \pa1-  \frac 12  (\varepsilon e^{-2\mu x_3}+\varepsilon+\mu^3 x_2^2)  \pa2+ \frac 12  \mu^2 x_2 \pa3,
\\[7pt]
\nabla_{\pa2}\pa3=\nabla_{\pa3}\pa2=\frac 12  \mu \pa2.
\end{array}
\end{equation}
Here, and throughout the paper,      $(\partial {_1},\partial {_2},\partial {_3})=(\frac{\partial}{\partial x_1},\frac{\partial}{\partial x_2},\frac{\partial}{\partial x_3})$ will denote the basis of coordinate vector fields.

\medskip
b) The Riemann curvature tensor $R$ of $g$ of type $(0,4)$  is completely determined (up to symmetries) by the following possibly non-vanishing components $R_{ijkh}=R(\partial_i,\partial_j,\partial_k,\partial_h)$:
\begin{equation}\label{Rgcoord}
\begin{array}{ll}
R_{1212}=\frac{1}{4}\mu ^2 , \quad &
R_{1313}=\frac{1}{4}\mu (\mu^3 x_2^2-5\varepsilon e^{-2\mu x_3}+\varepsilon ) ,
\\[7pt]
R_{1213}=\frac{1}{4}\mu ^3 x_2 , \quad &
R_{1323}=-\frac{1}{4}\mu ^2.
\end{array}
\end{equation}

c) The Ricci tensor $\varrho$ of $g$  is completely determined with respect to $\left\{ \pa {_i}\right\}$ by the following symmetric matrix $(\varrho_{ij})=   (\varrho(\partial _i, \partial _j))$:
\begin{equation}\label{Riccigcoord}
(\varrho_{ij})=\left(
\begin{array}{ccc}
  -\frac{1}{2}\varepsilon \mu (3e^{-2\mu x_3}-1) &   -\frac{1}{2}\mu ^2
  &  -\frac{1}{2}\mu ^3 x_2
\\[7pt]
-\frac{1}{2}\mu ^2  &
0 & 0
\\[7pt]
 -\frac{1}{2}\mu ^3 x_2 & 0 &  -\frac{1}{2} \mu ^2
\end{array}
\right).
\end{equation}
\end{theorem}

\begin{proof}
Consider the three-dimensional Lorentzian Lie group $(G=\widetilde{SL}(2,\mathbb R),g)$ described at the Lie algebra level by \eqref{lie} and \eqref{glie}. Then, there exists (at least, locally) some chart
$$
\begin{array}{rcl}
\sigma:  U\subseteq \mathbb R^3 & \to & G \\[3pt]
(x_1,x_2,x_3) & \mapsto & \exp(x_1u_1)\exp(x_2u_2)\exp(x_3u_3).
\end{array}
$$
With respect to coordinates $(x_1,x_2,x_3)$, the {\em Maurer-Cartan forms} are the components of $\sigma^{-1}d\sigma$, derived from the Lie algebra $\mathfrak g$ of $G$ by means of the relation
{\small
$$
\begin{array}{rl}
\sigma^{-1}d\sigma=&\exp(-x_3u_3)\exp(-x_2u_2)\exp(-x_1u_1)\cdot dx_1\cdot u_1\exp(x_1u_1)\exp(x_2u_2)\exp(x_3u_3)\\
&+\exp(-x_3u_3)\exp(-x_2u_2)\cdot dx_2\cdot u_2\exp(x_2u_2)\exp(x_3u_3)\\
&+\exp(-x_3u_3)\cdot dx_3\cdot u_3\exp(x_3u_3),
\end{array}
$$}
where
$$
\exp(-x_iu_i)u_j\exp(x_iu_i)=u_j-x_i[u_i,u_j]+\frac{x_i^2}{2!}[u_i,[u_i,u_j]]-\cdots .
$$
Explicitly, for the Lie algebra  \eqref{lie} equipped  with the left-invariant metric  \eqref{glie}, we obtain

$$
\begin{array}{l}
\exp(-x_3u_3)\exp(-x_2u_2)\exp(-x_1u_1)\cdot dx_1\cdot u_1\exp(x_1u_1)\exp(x_2u_2)\exp(x_3u_3)\\[7pt]
\quad=\exp(-x_3u_3)\exp(-x_2u_2)\cdot dx_1\cdot u_1\exp(x_2u_2)\exp(x_3u_3)\\[5pt]
\quad=\exp(-x_3u_3)\cdot dx_1\cdot \left(u_1+x_2\mu u_3-\frac{x_2^2}{2}\mu^2 u_2\right)\exp(x_3u_3)\\[7pt]
\quad=\big(u_1+x_2\mu u_3-\frac{x_2^2}{2}\mu^2 u_2-x_3(\mu u_1+(\frac12x_2^2\mu^3+\varepsilon)u_2)+\frac{x_3^2}{2!}(\mu^2 u_1-\frac12x_2^2\mu^4 u_2)\\[5pt]
\quad\quad-\frac{x_3^3}{3!}(\mu^3 u_1+(\frac12x_2^2\mu^5+\mu^2\varepsilon)u_2)+\cdots\big)dx_1\\[5pt]
\quad=\big(e^{-x_3\mu}dx_1\big)u_1-\big((\frac{x_2^2}{2}\mu^2e^{x_3\mu}+\frac{\varepsilon}{\mu}\sinh(x_3\mu))dx_1\big)u_2+\big(x_2\mu dx_1\big)u_3,\\[15pt]
\exp(-x_3u_3)\exp(-x_2u_2)\cdot dx_2\cdot u_2 \exp(x_2u_2)\exp(x_3u_3)\\[5pt]
\quad=\exp(-x_3u_3)\cdot dx_2\cdot u_2\exp(x_3u_3)\\[5pt]
\quad=\big(u_2+x_3\mu u_2+\frac{x_3^2}{2!}\mu^2 u_2+\frac{x_3^2}{3!}\mu^3 u_2+\cdots\big)dx_2=\big(e^{x_3\mu}dx_2\big)u_2,\\[15pt]
\exp(-x_3u_3)\cdot dx_3\cdot u_3\exp(x_3u_3)=dx_3u_3.
\end{array}
$$
From the above equations  we then explictly get the basis $\{ \theta_i \}$ of $1$-forms dual to the pseudo-orthonormal basis $\{u_i \}$ used in \eqref{glie} in local coordinates, that is,
$$
\left\{\begin{array}{l}
\theta_1=e^{-\mu x_3}dx_1,\\[5pt]
\theta_2={-\frac{1}{2\mu} \left( (\mu^3 {x_2^2}+\varepsilon ) e^{\mu x_3} -\e  e^{-\mu x_3})\right)dx_1}+e^{\mu x_3}dx_2,\\[5pt]
\theta_3=x_2\mu dx_1+dx_3.
\end{array}
\right.
$$
Consequently, we obtain the following explicit form of the pseudo-orthonormal basis $\{u_i \}$ with respect to the coordinate system $(x_1,x_2,x_3)$:
\begin{equation}\label{ui}
\left\{\begin{array}{l}
u_1=e^{\mu x_3} \pa1+{\frac{1}{2\mu} \left((\mu^3 {x_2^2}+\varepsilon ) e^{\mu x_3} -\e e^{-\mu x_3})\right)\pa2}
-\mu x_2 e^{\mu x_3} \pa3,\\[5pt]
u_2=e^{-\mu x_3} \pa2,\\[5pt]
u_3=\pa3.
\end{array}
\right.
\end{equation}
It is easy to check that the vector fields described by \eqref{ui} satisfy the bracket relations \eqref{lie}. The explicit form \eqref{g} of the left-invariant Lorentzian metric is then obtained applying  \eqref{glie}  to the explicit expressions \eqref{ui} of the basis of the Lie algebra (equivalently, substituting the explicit forms of $\theta _i$ into $g=\theta_1\theta_2+\theta_3^2$).

The proof of statements {\em a),b),c)}, which we shall omit, then follows directly from \eqref{g} via some standard calculations, applying the well know formulas for the local components of the Levi-Civita connection, the curvature tensor and the Ricci tensor.

\end{proof}

\begin{rem}\label{notEinstein}
{\em We observe that by \eqref{g} and \eqref{Riccigcoord} it easily follows that the metric $g$ is never Einstein. 
}
\end{rem}

\section{The Ricci soliton equation for the exceptional example}
\setcounter{equation}{0}

With respect to the coordinate system $(x_1,x_2,x_3)$ used in the previous Section, let now $X=X_i \pa i$ denote an arbitrary vector field on
$\widetilde{SL}(2,\mathbb R)$, where $X_i=X_i(x_1,x_2,x_3), \ i=1,2,3$, are some smooth functions. The Lie derivative
$\mathcal{L}_X g$ of the metric tensor is completely determined by $(\mathcal{L}_X g){(\pa i ,\pa j)}$, for all indices $i \leq j$. Explicitly, starting from \eqref{g},  we find
\begin{equation}\label{Lieg}
\begin{array}{l}
(\mathcal L _X g)_{11}=\frac{2}{\mu} \left(\varepsilon (e^{-2\mu x_3}-1) \pa1 X_1+\mu \pa1 X_2 +\mu^2x_2 \pa1 X_3-\varepsilon \mu  e^{-2\mu x_3}X_3  \right),
\\[7pt]
(\mathcal L _X g)_{12}= \frac{1}{\mu} \left(\mu  \pa1 X_1 +\varepsilon (e^{-2\mu x_3}-1) \pa2 X_1+\mu \pa2 X_2 +\mu^2x_2 \pa2 X_3  \right),
\\[7pt]
(\mathcal L _X g)_{13}= \mu x_2  \pa1 X_1 +\frac{\varepsilon}{\mu}  (e^{-2\mu x_3}-1) \pa3 X_1+\mu X_2 + \pa3 X_2 +\pa1 X_3 +\mu x_2 \pa3 X_3,
\\[7pt]
(\mathcal L _X g)_{22}= 2\pa2 X_1 ,
\\[7pt]
(\mathcal L _X g)_{23}= \mu x_2  \pa2 X_1 + \pa3 X_1+ \pa2 X_3,
\\[7pt]
(\mathcal L _X g)_{33} = 2 \mu x_2 \pa3 X_1+2\pa3 X_3.
\end{array}
\end{equation}
Using the above components of $\mathcal{L}_X g$ and  \eqref{g}, \eqref{Riccigcoord}, we obtain that $g$, a real constant $\lambda$ and the smooth vector field $X$ satisfy  the Ricci soliton equation \eqref{solit} if and only if the following system of six PDE is satisfied:
\begin{equation}\label{RSg}
\left\{
\begin{array}{l}
\frac{1}{2\mu} \left(4\varepsilon (e^{-2\mu x_3}-1) \pa1 X_1+4\mu \pa1 X_2 +4\mu^2x_2 \pa1 X_3-4\varepsilon \mu  e^{-2\mu x_3}X_3 \right. \\[3pt]
\left. \qquad -3\varepsilon\mu^2e^{-2\mu x_3}+\varepsilon \mu^2+2\varepsilon \lambda (e^{-2\mu x_3}-1)   \right)=0,
\\[7pt]
 \frac{1}{2\mu} \left(2\mu  \pa1 X_1 +2\varepsilon (e^{-2\mu x_3}-1) \pa2 X_1+2\mu \pa2 X_2 +2\mu^2x_2 \pa2 X_3-\mu (\mu^2+2\lambda)  \right)=0,
\\[7pt]
\frac{1}{2\mu} \left(2\mu^2 x_2  \pa1 X_1 +2\varepsilon (e^{-2\mu x_3}-1) \pa3 X_1+2\mu^2 X_2 +2\mu  \pa3 X_2 +2\mu \pa1 X_3 \right. \\[3pt] \left. \qquad +2\mu^2x_2 \pa3 X_3 -\mu^2 (\mu^2+2\lambda)x_2 \right)=0,
\\[7pt]
2\pa2 X_1 =0,
\\[7pt]
 \mu x_2  \pa2 X_1 + \pa3 X_1+ \pa2 X_3=0,
\\[7pt]
2 \mu x_2 \pa3 X_1+2\pa3 X_3 -\frac 12 \mu^2-\lambda=0.
\end{array}
\right.
\end{equation}
We integrate the fourth equation of~\eqref{RSg} and we get
\begin{equation}\label{gX1}
\begin{array}{l}
X_1 =  {F_1} \left( x_{{1}},x_{{3}} \right),
\end{array}
\end{equation}
for some smooth  function $F_1$. We then substitute from \eqref{gX1} into the last equation of~\eqref{RSg} and by integration we get
\begin{equation}\label{gX3}
\begin{array}{l}
X_3 =    -\mu x_2 F_1(x_1,x_3) + \frac{1}{4}  (\mu^2+\lambda ) x_3 +{F_3} \left( x_{{1}},x_{{2}} \right),
\end{array}
\end{equation}
for some smooth  function $F_3$. Using \eqref{gX1} and  \eqref{gX3} in the second equation and integrating, we find
\begin{equation}\label{gX2}
\begin{array}{l}
X_2 = \frac 12 \mu^2 x_2^2 F_1  - \mu x_2 {F_3} +\mu \int {F_3} dx_2-\left(\pa1 F_1-\frac 12 \mu^2-\lambda \right)x_2     +{F_2} \left( x_{{1}},x_{{3}} \right),
\end{array}
\end{equation}
where $F_2$ is a smooth fuction.  The fifth equation of \eqref{RSg} now reads
\begin{equation} \label{fifth}
\begin{array}{l}
\pa3 F_1(x_1,x_3)-\mu F_1(x_1,x_3) +\pa2 F_3(x_1,x_2)=0.
\end{array}
\end{equation}
We differentiate the above equation with respect to $x_2$ and we find $\partial_{22}^2 F_3=0$, whose integration yields
\begin{equation} \label{F3}
\begin{array}{l}
F_3(x_1,x_2) =G_3(x_1)x_2+H_3(x_1),
\end{array}
\end{equation}
where $G_3,H_3$ are some smooth functions. Equation \eqref{fifth} now reads
$$
\begin{array}{l}
\pa3 F_1(x_1,x_3)-\mu F_1(x_1,x_3) +G_3(x_1)=0
\end{array}
$$
whose integral is given by
\begin{equation} \label{F1}
\begin{array}{l}
F_1(x_1,x_3)= \frac{1}{\mu} G_3(x_1) +e^{\mu x_3}G_1(x_1),
\end{array}
\end{equation}
where $G_1$ is a smooth function. Moreover, substituting from \eqref{F3}, the expression \eqref{gX2} of $X_2$ now simplifies as follows:
\begin{equation}\label{gX2'}
\begin{array}{l}
X_2 = \frac{1}{ 2\mu} \left(  \mu^3 x_2^2 e^{\mu x_3} G_1(x_1)  -2 \mu x_2 e^{2\mu x_3} G_1' (x_1) -2 x_2 G_3' (x_1)
+\mu(\mu^2  +2 \lambda )x_2 \right)   \\[3pt]
\qquad +{F_2} \left( x_{{1}},x_{{3}} \right).
\end{array}
\end{equation}
After we substitute from the above equations, the third equation in \eqref{RSg}, written as a polynomial equation in the variable $x_2$, reads
\begin{equation}\label{third}
\begin{array}{l}
\left\{-2\mu e^{\mu x_3}G_1' (x_1)+\frac 14 \mu (2\mu^2+\lambda)\right\} x_2 \\[3pt]
+ \left\{ \mu F_2(x_1,x_3)+\pa3 F_2(x_1,x_3) +H'_3 (x_1) +\e (e^{-\mu x_3}-e^{\mu x_3}) G_1(x_1)\right\}=0.
\end{array}
\end{equation}
The above equation \eqref{third} must hold for all values of $x_2$. Therefore, in particular it implies
$$
-2\mu e^{\mu x_3}G_1' (x_1)+\frac 14 \mu (2\mu^2+\lambda)=0,
$$
for all values of $x_3$, whence,

\begin{equation}\label{lambdaG1}
\begin{array}{l}
G_1 (x_1) =   c_1 , \qquad  \lambda = -\frac 12 \mu^2 <0,
\end{array}
\end{equation}
for some real constant $c_1$.  Consequently, equation \eqref{third} reduces to
\begin{equation}\label{third1}
\begin{array}{l}
\mu F_2(x_1,x_3)+\pa3 F_2(x_1,x_3) +H'_3 (x_1) +\e (e^{-\mu x_3}-e^{\mu x_3}) c_1=0,
\end{array}
\end{equation}
which, by integration, yields
\begin{equation}\label{F2}
\begin{array}{l}
F_2  (x_1,x_3) = \left( -\e c_1 x_3+\frac{\e}{2\mu}c_1 e^{2\mu x_3} -\frac{1}{\mu} e^{\mu x_3}H'_3(x_1)+G_2(x_1)  \right)e^{-\mu x_3} ,
\end{array}
\end{equation}
where $G_2$ is a smooth function. We are now left with the first equation in \eqref{RSg}, which, written as a polynomial in $x_2$, reads
\begin{equation}\label{first}
\begin{array}{l}
-\frac{2}{\mu} \left(   G''_3(x_1)-\mu^2  H'_3(x_1)-\e \mu^2 c_1 e^{-\mu x_3} \right) x_2\\[5pt]
-\frac{1}{\mu} \left\{ (2\e \mu^2 H_3(x_1)-2\e G'_3(x_1)+\e \mu ^3  )e^{-2\mu x_3}\right\}\\[5pt]
-\frac{1}{\mu}	\left\{2 \mu^2 G'_2(x_1)e^{-\mu x_3}-2\mu H''_3(x_1)- 2\e G'_3(x_1) \right\} =0,
\end{array}
\end{equation}
From \eqref{first}, which must be satisfied for all values of $x_2$, in particular we get 
$$
G''_3(x_1)-\mu^2  H'_3(x_1)-\e \mu^2 c_1 e^{-\mu x_3}=0,
$$
for all values of $x_3$. Therefore, $c_1=0$ and
\begin{equation}\label{H3}
\begin{array}{l}
H_3  (x_1) =\frac{1}{\mu^2} G'_3 (x_1)+c_2 ,
\end{array}
\end{equation}
where $c_2$ is a real constant. Equation \eqref{first} then reduces to
\begin{equation}\label{first1}
\begin{array}{l}
\e \mu^3 (2  c_2+\mu  )e^{-2\mu x_3}-2 \mu^3 G'_2(x_1)e^{-\mu x_3}+2(G'''_3(x_1)+\e \mu G'_3(x_1)) =0,
\end{array}
\end{equation}
which must hold for all values of $x_3$ So, it  yields at once $c_2=-\frac{\mu}{2}$ and $G_2 (x_1)=k_4$ is a real constant. At this point, \eqref{first1} reduces to
$$
G'''_3(x_1)+\e \mu G'_3(x_1)=0,
$$
whose general integral is the following:
\begin{equation}\label{G3}
G_3  (x_1) =
\left\{
\begin{array}{ll}
k_1 e^{\sqrt{-\varepsilon \mu} x_1}+k_2 e^{-\sqrt{-\varepsilon \mu} x_1}+k_3 & \text{if} \;\, \e\mu <0, \\[5pt]
k_1 \cos (\sqrt{\varepsilon \mu} x_1) +k_2 \sin (\sqrt{\varepsilon \mu} x_1) +k_3 & \text{if} \;\, \e\mu> 0.
\end{array}
\right.
\end{equation}
Substituting the above expressions into $X_1,X_2,X_3$, we find  either \eqref{Xig1} or \eqref{Xig2}, depending on whether $\e \mu <0$ or $\e \mu >0$, respectively. All equations in \eqref{RSg} are now satisfied. In fact, as a further check, using \eqref{g},\eqref{Riccigcoord},\eqref{Lieg} and either \eqref{Xig1} or \eqref{Xig2}, it is easily seen that
$(\mathcal{L}_X g)_{ij}=-\frac 12 \mu^2 g_{ij}-\varrho_{ij}$, for all indices $i,j=1,2,3$.

\begin{rem}
{\em
We explicitly observe that the expressions of $X_i$ given in \eqref{Xig1} and in \eqref{Xig2} depend on four arbitrary real constants $k_1,k_2,k_3,k_4$. This reflects the fact that the Lie algebra of Killing vector fields of $(\widetilde{SL}(2,\mathbb R), g)$ is four-dimensional, coherently to the fact that the isometry algebra of $(\widetilde{SL}(2,\mathbb R), g)$ has dimension four.
}
\end{rem}

By Remark \ref{notEinstein},  the above Ricci solitons are not trivial, because the metric \eqref{g} is never Einstein.  In order to complete the proof of Theorem \ref{main}, we shall now prove that these Ricci solitons are not gradient. We shall give the full details for the case where $\e\mu < 0$, the case $\e \mu >0$ can be treated in a completely analogous way.

So, consider a smooth vector field $X =X_i \partial_i$, with $X_i$ described by \eqref{Xig1}, and suppose that such a vector field is gradient, that is, there exists some smooth function $f=f(x_1,x_2,x_3)$, such that $X={\rm grad} _g (f)$. Then, by \eqref{g} and \eqref{Xig1}, the function $f$ must satisfy the following system of PDE:
\begin{equation}\label{GRS}
\left\{
\begin{array}{l}
  k_1 e^{\sqrt{-\varepsilon \mu} x_1}+k_2 e^{-\sqrt{-\varepsilon \mu} x_1}+k_3 = \mu\pa2 f ,\\[7pt]
  - \frac{1}{\mu^2}   \left( k_1(\mu\sqrt{-\varepsilon\mu} x_2 -\varepsilon) e^{\sqrt{-\varepsilon \mu} x_1} -k_2(\mu\sqrt{-\varepsilon\mu} x_2 +\varepsilon) e^{-\sqrt{-\varepsilon \mu} x_1} \right) +k_4 e^{-\mu x_3}  \\[5pt]
  =    \pa1 f-\frac{\e}{\mu} (e^{-2\mu x_3}+1)  \pa2 f+\mu^2 x_2^2  \pa2 f-\mu x_2 \pa3 f,\\[3pt]
 \sqrt{-\varepsilon \mu}    \left(k_1 e^{\sqrt{-\varepsilon \mu} x_1}-k_2 e^{-\sqrt{-\varepsilon \mu} x_1} \right) ={\mu^2} \left( \mu \pa2 f-\pa3 f-\frac{\mu}{2}\right).
\end{array}
\right.
\end{equation}
Integrating the first equation in \eqref{GRS}  we get
\begin{equation}\label{effe}
\begin{array}{l}
f(x_1,x_2,x_3) = \frac{1}{\mu} \left(  k_1 e^{\sqrt{-\varepsilon \mu} x_1}+k_2 e^{-\sqrt{-\varepsilon \mu} x_1}+k_3 \right)x_2 +Q(x_1,x_3) ,
\end{array}
\end{equation}
where $Q$ denotes an arbitrary smooth function. We substitute from \eqref{effe} into the third equation of \eqref{GRS} and we obtain
\begin{equation}\label{GRSIII}
\begin{array}{l}
k_1 (\mu^2 x_2+ \sqrt{-\varepsilon \mu}) e^{\sqrt{-\varepsilon \mu} x_1}  +k_2 (\mu^2 x_2 -\sqrt{-\varepsilon \mu}) e^{-\sqrt{-\varepsilon \mu} x_1}\\[3pt]
+{\mu^2} \left(k_3 x_2 -\pa3 Q - \frac{\mu}{2} \right)=0.
\end{array}
\end{equation}
As the above equation must be satisfied for all values of $x_1$ and $x_2$, it yields at once $k_1=k_2=k_3=0$ and $\pa3 Q = \frac{\mu}{2}$, whence, by integration, we find
\begin{equation}\label{qu}
\begin{array}{l}
Q(x_1,x_3)= -\frac{1}{2} \mu x_3 +R(x_1) ,
\end{array}
\end{equation}
for some smooth function $R$. Using \eqref{qu} and $k_1=k_2=k_3=0$, the second equation of \eqref{GRS} reduces to
\begin{equation}\label{GRSII}
\begin{array}{l}
R'(x_1)+\frac{1}{2}\mu^2 x_2=0,
\end{array}
\end{equation}
which cannot hold for all values of $x_2$, as $\mu \neq 0$. Therefore, the Ricci soliton is never gradient.

\end{document}